\documentclass[12pt]{amsart}
\usepackage{amssymb,verbatim,enumerate,ifthen}
\usepackage{esint}
\usepackage[mathscr]{eucal}
\usepackage[utf8]{inputenc}
\usepackage[T1]{fontenc}
\usepackage{pgf,tikz,pgfplots}
\pgfplotsset{compat=1.13}
\usepackage{mathrsfs}
\usetikzlibrary{arrows}

\usepackage{enumerate,ifthen}
\usepackage{enumitem}

\makeatletter
\@namedef{subjclassname@2010}{%
  \textup{2010} Mathematics Subject Classification}
\makeatother

\newtheorem{thm}{Theorem}
\newtheorem{cor}{Corollary}
\newtheorem{prop}{Proposition}
\newtheorem{lem}{Lemma}
\numberwithin{lem}{section}

\newtheorem{exa}{Example}

\numberwithin{equation}{section}

\newcommand\R{\mathbb{R}}

\newcommand\N{\mathbb{N}}
\newcommand\Q{\mathbb{Q}}
\newcommand\Z{\mathbb{Z}}
\newcommand{\QA}[1]{A^{[#1]}}

\newcommand{\Sm}{\mathcal{S}}
\newcommand{\LDer}{\underline{\mathcal{D}}}
\DeclareMathOperator{\CM}{\mathcal{CM}}

\numberwithin{equation}{section}
\def\eq#1{{\rm(\ref{#1})}}
\def\Eq#1#2{\ifthenelse{\equal{#1}{*}}
  {\begin{equation*}\begin{aligned}[]#2\end{aligned}\end{equation*}}
  {\begin{equation}\begin{aligned}[]\label{#1}#2\end{aligned}\end{equation}}}
\newcommand{\calF}{\mathcal{F}}
\newcommand{\UQA}[1]{\mathscr{U}_{#1}}
\newcommand{\LQA}[1]{\mathscr{L}_{#1}}
\title{On a lattice-like property of quasi-arithmetic means}

\author[P. Pasteczka]{Pawe\l{} Pasteczka}
\address{Institute of Mathematics \\ Pedagogical University of Cracow \\ Podchor\k{a}\.zych str. 2, 30-084 Krak\'ow, Poland}
\email{pawel.pasteczka@up.krakow.pl}
\begin{document}

\begin{abstract}
 We will prove that in a family of quasi-arithmetic means sattisfying certain smoothness assumption (embed with a naural pointwise ordering) every finite family has both supremum and infimum, which is also a quasi-arithmetic mean sattisfying the same smoothness assumptions. More precisely, if $f$ and $g$ are $\mathcal{C}^2$ functions with nowhere vanishing first derivative then there exists a function $h$ such that: (i) $A^{[f]} \le A^{[h]}$, (ii) $A^{[g]} \le A^{[h]}$, and (iii) for every continuous strictly monotone function $s \colon I \to \mathbb{R}$
 $$
A^{[f]} \le A^{[s]} \text{ and } A^{[g]} \le A^{[s]} \text{ implies } A^{[h]} \le A^{[s]}
 $$
($A^{[f]}$ stands for a quasi-arithmetic mean generated by a function $f$ and so on).
 Moreover $h\in\mathcal{C}^2$, $h'\ne0$, and it is a solution of the differential equation
$$
\frac{h''}{h'}=\max\Big(\frac{f''}{f'},\,\frac{g''}{g'}\Big).
$$
We also provide some extension to a finite family of means.

Obviously dual statements with inverses inequality sign as well as a multifuntion generalization will be also stated.
\end{abstract}

\maketitle

\section{Introduction}
Quasi-arithmetic means were introduced as a generalization of Power Means in 1920s/30s in a series of nearly simultaneous papers \cite{Kno28,Def31,Kol30,Nag30}.  
For an interval $I$ and a continuous and strictly monotone function $f \colon I \to \R$ (from now $I$ stands for an interval and $\CM$ is a family of continuous, strictly monotone functions on $I$) we define \emph{quasi-arithmetic mean} $\QA{f} \colon \bigcup_{n=1}^\infty I^n \to I$  by 
\Eq{*}{
\QA{f}(v):=f^{-1}\left( \frac{f(v_1)+f(v_2)+\cdots+f(v_n)}{n} \right), \quad (n \in \N,\: v \in I^n)\:.
}
The function $f$ is called a \emph{generator} of quasi-arithmetic mean.

It it well known that for
$I=\R_+$, $\pi_p(x):=x^p$ for $p\ne 0$ and $\pi_0(x):=\ln x$, then mean $\QA{\pi_p}$ coincides with the $p$-th power mean (this fact had been already noticed by Knopp \cite{Kno28}).

There were a number of results related to quasi-arithmetic means. For example one can define  the preorder on $\CM$ as follows: 
\Eq{*}{
f \prec g \iff \QA{f}(v) \le \QA{g}(v) \text{ for all } v \in \bigcup_{n \in \N}I^n.
}
It is well know (see for example \cite{HarLitPol34} ) that $\QA{f}=\QA{g}$ if and only if their generators are affine transformation of each other, i.e. there exist $\alpha,\,\beta \in \R$ such that $g=\alpha f + \beta$. Therefore it is natural to define relation $\sim$ on $\CM$ by
\Eq{*}{
f \sim g \iff \QA{f} = \QA{g}\:.
}

Furthermore it is easy to check that $\prec$ induces a partial order on $\CM /_\sim$\:. This order has a lot of interesting properties (see for example results by Cargo-Shisha \cite{CarShi64,CarShi69} and by the author \cite{Pas1710}).

As we are going to elaborate some lattice properties of quasi-arithmetic means, we need to introduce supremum and infimum of any subset in this family.
First of all, when we have one element only, we can naturally define the set of all functions generating a bigger quasi-arithmetic mean $U_f:=\{ s \in \CM \colon f \prec s\}$. 
Then, for every subfamily $\calF \subset \CM$ we define a function $\UQA{\calF} \colon \bigcup_{n=1}^\infty I^n \to I$ by
\Eq{*}{
\UQA{\calF}(v) := 
\begin{cases}
\inf \big\{\QA{s}(v) \colon s \in \bigcap\limits_{f \in \calF} U_f\big\} & \text{ if }\bigcap\limits_{f \in \calF} U_f \ne \emptyset,\\[5mm]
\max(v) & \text{ otherwise.}
\end{cases}
}
Notice that $\UQA{\calF}$ is a mean, while $U_f$ is a family of functions. Similarly one can define $L_f:=\{ s \in \CM \colon f \succ s\}$ and, for $\calF \subset \CM$, a mean $\LQA{\calF} \colon \bigcup_{n=1}^\infty I^n \to I$ by
\Eq{*}{
\LQA{\calF}(v) := 
\begin{cases}
\sup \big\{\QA{s}(v) \colon s \in \bigcap\limits_{f \in \calF} L_f\big\} & \text{ if }\bigcap\limits_{f \in \calF} L_f \ne \emptyset,\\[5mm]
\min(v) & \text{ otherwise.}
\end{cases}
}

Obviously for every $\calF$ both $\LQA{\calF}$ and $\UQA{\calF}$ are monotone and symmetric. Furthermore whenever they are quasi-arithmetic means, then generatiors of $\LQA \calF$ and $\UQA \calF$ are infimum and supremum of $\calF$, respectively (with respect to the partial ordering $\prec$). Therefore it is very natural to ask about possible sufficient conditions to $\LQA\calF$ and $\UQA{\calF}$ to be a quasi-arithmetic mean and, knowing this, about its generator.

In the present paper we will focus on a family of means generated by $\mathcal{C}^2(I)$ functions having nowhere vanishing first derivative (from now on we will denote a family of all such functions by $\Sm$\,). Our main results states that whenever we have two such functions form $\Sm$ then they have both supremum and infimum. These supremum and infimum are quasi-arithmetic means generated by functions from $\Sm$. Furthermore we present a formulae to calculate them in term of ordinary diffential equations (equalities on Arrow-Pratt indexes).

\subsection{Properties of quasi-arithmetic means}
In this section we will present some selection of known results concerning quasi-arithmetic means. It fact there exists a rich literature in this area (see for example \cite[chap. 4]{Bul03} and references therein) and, as a natural consequence, we will present just these results which will be used later.

First of all we list some known comparability conditions for quasi-arithmetic means.

\begin{prop}
\label{prop:comp}
 Let $f,\,g \in \CM$. Then $f\prec g$ if and only if one of the following (equivalent) conditions are valid
 \begin{enumerate}[label={(\roman*)},series=theoremconditions]
\item \label{CC:gf^-1} $g$ is increasing and $g \circ f^{-1}$ is convex or $g$ is decreasing and $g \circ f^{-1}$ is concave,
\item \label{CC:fg^-1} $f$ is increasing and $f \circ g^{-1}$ is concave or $f$ is decreasing and $f \circ g^{-1}$ is convex.
\end{enumerate}
Moreover if $f$ and $g$ are both differentiable and $f \cdot g' \ne 0$ we obtain equvalent statements
 \begin{enumerate}[resume*=theoremconditions]
  \item\label{CC:g'/f'} $f$ and $g$ are of the same monotonicity and $g'/f'$ is nondecreasing or $f$ and $g$ are of the opposite monotonicity and $g'/f'$ is nonincreasing; 
  \item\label{CC:f'/g'} $f$ and $g$ are of the same monotonicity and $f'/g'$ is nonincreasing or $f$ and $g$ are of the opposite monotonicity and $f'/g'$ is nondecreasing.
  \end{enumerate}
Additionally if $f$ and $g$ are twice differentiable we have 
 \begin{enumerate}[resume*=theoremconditions]
 \item \label{CC:Mikusinski} $\tfrac{f''(x)}{f'(x)} \le \tfrac{g''(x)}{g'(x)}$ for all $x \in I$.
 \end{enumerate}
\end{prop}

Let us stress two notions which are eqivalent to all conditions above: $g \in U_f$ and $f \in L_g$. Formally, these two are not conditions of comparability, and therefore they are not listed among others.

Notice that condition \ref{CC:Mikusinski} is defined for functions in $\Sm$ only. Now let us recall some recent results concerning smoothness properties implied by comparability enclosed in \cite{Pas1710}.

\begin{lem}[\cite{Pas1710}, Theorem 7]
 \label{lem:hnonvanish}
 Let $f,\,g \in \CM$ such that $f\prec g$. 
 Then 
 \begin{itemize}
  \item one-sided derivatives $f'_-$ and $g'_-$  (resp. $f'_+$ and $g'_+$) exist at the same points;
  \item $f'_-$ and $g'_-$ (resp. $f'_+$ and $g'_+$) vanish at the same points.
 \end{itemize}
\end{lem}

As a simple but useful conclusion, which was not worded in \cite{Pas1710}, we get the following corollary
\begin{cor}
 \label{cor:hnonvanish}
 Let $f,\,g \in \CM$ be two differentiable, strictly monotone functions such that $f \prec g$. Then
 \Eq{*}{
 \big\{x \in I \mid f'(x)=0\big\}= \big\{x \in I \mid g'(x)=0\big\}\:.
 }
\end{cor}

Furthermore in view of \cite[Theorem 12]{Pas1710} and Corollary~\ref{cor:hnonvanish} above we can establish the following lemma. 

\begin{lem}
 \label{lem:diffnonz}
  Let $f,\,g,\,h \in \CM$ such that $f \prec g \prec h$.
  If $f$ and $h$ are both differentiable at some point $x_0 \in I$, then so is $g$.
  
  Moreover all equalities $f'(x_0)=0$, $g'(x_0)=0$, and $h'(x_0)=0$ are pairwise equivalent.
\end{lem}

At the end of this section let us recall two results concerning convergence of means. First of them was obtained in 1990s by P\'ales \cite{Pal91}. 
\begin{lem}[\cite{Pal91}, Corollary~1]
 \label{lem:Pales}
Let $f$ and $f_n$, $n \in \N$, be continuous, strictly monotone functions defined on $I$. Then $\lim\limits_{n \to \infty} \QA{f_n}=\QA{f}$ pointwise if and only if
\Eq{*}{
\lim_{n \to \infty} \frac{f_n(x)-f_n(z)}{f_n(y)-f_n(z)}=\frac{f(x)-f(z)}{f(y)-f(z)} \qquad \text{for all }x,\,y,\,z \in I \text{ with } y \ne z.
}
\end{lem}

Second one was obtained by the author in 2013 \cite{Pas13} and strenghtened in both 2015 \cite{Pas15b} and 2018 \cite{Pas18a}, for the sake of brevity we will use the 2013's version. 
\begin{lem}[\cite{Pas13}, Corollary 3]
\label{lem:Pas13} Let $f$ and $f_n$ for $n \in \N$ be a functions from $\Sm$.
If $f_n''/f_n' \to f''/f'$ in $L^1(I)$ then $\QA{f_n} \to \QA{f}$ uniformly.
\end{lem}

The rest of the present note is organized as follows -- in the next section we formulate our main result and present some of its applications. Its long proof  (jointly with all relevant technicalities) is shifted to the last section (section~\ref{sec:proofmain}).

\section{\label{sec:mainres} Main result}
It is known (see \ref{CC:Mikusinski}) that (under certain smoothness assumptions) comparability of quasi-arithmetic means are closely related with the so-called Arrow-Pratt index, i.e. the operator $f \mapsto f''/f'$. Moreover it is definitelly the easiest expression among all known conditions, as it reduces comparability of quasi-arithmetic means to comparability of single-variable functions associated with them. Moreover it can be proved (directly from \ref{CC:Mikusinski}) that
\Eq{*}{
f \sim g\iff \frac{f''}{f'}=\frac{g''}{g'}\:.
}
Thus quasi-arithmetic means are generalized rather by their Arrow-Pratt indexes than generators. 

We will establish, in terms of Arrow-Pratt index, the value of $\UQA{}$ for two means (i.e $\UQA{\{f,g\}}$, which will be alternatively denoted as $\UQA{f,g}$). 
\begin{lem}
\label{lem:main}
 For every $f,\,g \in \Sm$ there exists $h \in \Sm$ such that $\UQA{f,g}=\QA{h}$.  
 Moreover 
 \Eq{def:h}{
 \frac{h''}{h'}=\max\Big(\frac{f''}{f'},\,\frac{g''}{g'} \Big).
 }
\end{lem}

Let us stress that there are two important statements which are bind in this lemma. First of all it states that the value of $\UQA{f,g}$ is a quasi-arithmetic mean. Second, this quasi-arithmetic mean is generalized by a $\mathcal{C}^2$ function wiht nowhere vanishing derivatives. Having this proved, equality \eq{def:h} is an immediate corollary of \ref{CC:Mikusinski}.

Let us now turn to a slight generalization of Lemma~\ref{lem:main} which is in fact a main result in this paper. As Lemma~\ref{lem:main} deals with two means only, there appear a natural question -- what happens when a family $\calF$ is bigger. This generalization is presented in our main theorem.

\begin{thm}
 Let $\calF \subset \Sm$. 
 \begin{enumerate}
 \item If the function $G \colon I \ni x \mapsto \sup_{f \in \calF} \frac{f''(x)}{f'(x)}$ is continuous then $\UQA \calF=\QA{g}$, where $g \in \Sm$ and $g''/g'=G$.
  \item If the function $H \colon I \ni x \mapsto \inf_{f \in \calF} \frac{f''(x)}{f'(x)}$ is continuous then $\LQA \calF=\QA{h}$, where $h \in \Sm$ and $h''/h'=H$.
  \end{enumerate}
 \end{thm}

 \begin{proof}
By \cite[Corollary 2.1]{CarShi69} we know that the family of quasi-arithmetic means defined on $I$ is a separable space. Therefore we may assume that $\calF$ is countable, i.e. $\calF=(f_n)_{n \in \N}$.
Here and below we do not claim all $f_n$-s to be different, in particular this proof covers a case where $\calF$ is a finite family.

Define $M_n:=\UQA{f_1,\dots,f_n}$ and a sequence of functions $(g_n)_{n \in \N}$ from $\Sm$ such that 
\Eq{*}{
\frac{g_n''}{g_n'}=\max_{i \in \{1,\dots,n\}} \frac{f_i''}{f_i'} \quad (n \in \N).
}
We will prove by induction that $M_n=\QA{g_n}$ for all $n \in \N$. Indeed, for $n=1$ this equality is trivial. 

Having this we know that $M_{n+1} \ge M_n$ and $M_{n+1} \ge \QA{f_{n+1}}$. It implies $M_{n+1} \ge \UQA{g_n,f_{n+1}}$. Which, in view of Lemma~\ref{lem:main} implies $M_{n+1} \ge \QA{g_{n+1}}$. On the other hand 
\Eq{*}{
\QA{g_{n+1}} \ge \QA{f_k}\quad \text{ for all } k \in \{1,\dots,n+1\},
}
what implies $M_{n+1}=\QA{g_{n+1}}$. In particular $\UQA{\calF} \ge M_{n}$ for all $n \in \N$. Thus $\UQA{\calF} \ge \lim_{n \to \infty} M_n=\lim_{n \to \infty} \QA{g_n}$.

On the other hand, as $\QA{g} \ge \QA{f_n}$ for all $n \in \N$, we have $\UQA{\calF} \le \QA{g}$. Therefore
\Eq{*}{
\lim_{n \to \infty} \QA{g_n} \le \UQA{\calF} \le \QA{g}.
}
Finally in view of Lemma~\ref{lem:Pas13} we get $\lim_{n \to \infty} \QA{g_n}=\QA{g}$ and, consequently, $\UQA{\calF}=\QA{g}$.

To establish a second part we need to use a reflection of quasi-arithmetic means which were considered for example in \cite{PalPas18b}. For a given continuous, monotone function $f \colon I \to \R$ let $\hat f \colon (-I) \to \R$ be defined as $\hat f(x)=f(-x)$. Then, for every vector $v \in \bigcup_{n=1}^\infty I^n$ we have $\QA{f}(v)=-\QA{\hat f}(-v)$. As a trivial consequence, for every $f,\,g \colon I \to \R$, inequality $f \prec g$ is equivalent to $\hat f \succ \hat g$. In particular if we define 
$\widehat \calF:=\{\hat f \mid f \in \calF\}$, then 
\Eq{E:hat1}{
-\LQA{\calF}(-v)=\UQA{\widehat \calF}(v)=\QA{h_0}(v),
}
where by the first part $h_0 \colon I \to \R$ is a solution of a differential equation
\Eq{*}{
\frac{h_0''(x)}{h_0'(x)}=\sup_{f \in \widehat \calF} \frac{f''(x)}{f'(x)}=\sup_{f \in \calF} \frac{\hat f''(x)}{\hat f'(x)},\quad x \in (-I).
}
But $\hat u''(x)/\hat u'(x)=-u''(-x)/u'(-x)$ for all $u \in \Sm$. Additionally, in view of \eq{E:hat1}, $\LQA{\calF}=\QA{\hat h_0}$. Finally, for all $x \in I$,
\Eq{*}{
\frac{\hat h_0''(x)}{\hat h_0'(x)}
=-\frac{h_0''(-x)}{h_0'(-x)}
=-\sup_{f \in \calF} \frac{\hat f''(-x)}{\hat f'(-x)}
=\inf_{f \in \calF} \Big(-\frac{\hat f''(-x)}{\hat f'(-x)}\Big)
=\inf_{f \in \calF} \Big(\frac{f''(x)}{f'(x)}\Big),
}
what concludes the proof.
\end{proof}

Notice that Lemma~\ref{lem:main} implies that the set $\Sm/_\sim$ embedded with a partial ordering $\prec$ admit some type of lattice property. This statement has an alternative (significantly simpler) proof using \ref{CC:Mikusinski}. However, contrary to \ref{CC:Mikusinski}, Lemma~\ref{lem:main} considers the smallest upper bound of \emph{all} continuous and monotone functions, not only these from $\Sm$. 

\section{Examples}
First, let us applied a result in Lemma~\ref{lem:main} in a simple example

\begin{exa}
 Let $I=(-\tfrac{\pi}2,\tfrac\pi{2})$, $f,\,g \colon I \to \R$ be given by $f(x)=\sin(x)$, $g(x)=\tan x$. 
 
 Then $f''(x)/f'(x)=-\tan x$, and $g''(x)/g'(x)=2\tan x$. By Lemma~\ref{lem:main} we get, as $\tan$ is even, that $\UQA{f,g}=\QA{h}$, where
 \Eq{*}{
 h''(x)/h'(x)=
 \begin{cases} 
 f''(x)/f'(x) & \text{ for }x \in (-\tfrac{\pi}2,0] \\
 g''(x)/g'(x) & \text{ for }x \in (0,\tfrac\pi{2}).
 \end{cases}
 }
Therefore $h$ is $\mathcal{C}^2$ function, which is an affine transformation of $\sin$ and $\tan$ on negative and positive elements, respectively. One can easy prove that 
\Eq{*}{
 h(x)=
 \begin{cases} 
 \sin(x) & \text{ for }x \in (-\tfrac{\pi}2,0] \\
 \tan(x) & \text{ for }x \in (0,\tfrac\pi{2}).
 \end{cases}
 }
Similarly $\LQA{f,g}=\QA{k}$, where 
\Eq{*}{
 k(x)=
 \begin{cases} 
 \tan(x) & \text{ for }x \in (-\tfrac{\pi}2,0] \\
 \sin(x) & \text{ for }x \in (0,\tfrac\pi{2}).
 \end{cases}
 }
 Observe that both $h$ and $k$ are $\mathcal{C}^2$ functions (otherwise we should make an affine tranformation in a merging point).
\end{exa}

Second, we show that the assumption that first derivative is nonvanishing is important.

\begin{exa}
 Let $f,\,g \colon (-1,1) \to \R$ be given by $f(x)=x$ and $g(x)=x^3$. Then $\UQA{f,g}=\max$ and $\LQA{f,g}=\min$.

 Indeed, in view of Lemma~\ref{lem:hnonvanish}, for every $h \colon (-1,1) \to \R$ such that $f \prec h$ we get $h'_+(0)$ exists and $h'_+(0) \ne 0$. 
If we apply this lemma again we obtain that $\QA{g}$ and $\QA{h}$ are incomparable. Therefore $U_f \cap U_g=\emptyset$ and, consequently, $\UQA{f,g}=\max$. Prove of the second eqaulity is analogous. 
\end{exa}

\section{\label{sec:proofmain} Proof of Lemma~\ref{lem:main}}
All this section will be devoted to prove the central lemma in the present paper. It is divided into three parts.
Relatively short proof of this statement will be presented in section~\ref{ssec:thm:main}. First, we introduce the notion of bilateral derivative. Having this, we present vary auxiliary result which will be useful in a main part. Sketch of the proof, as well as its division into lemmas is ilustrated on Figure~\ref{fig:1}.

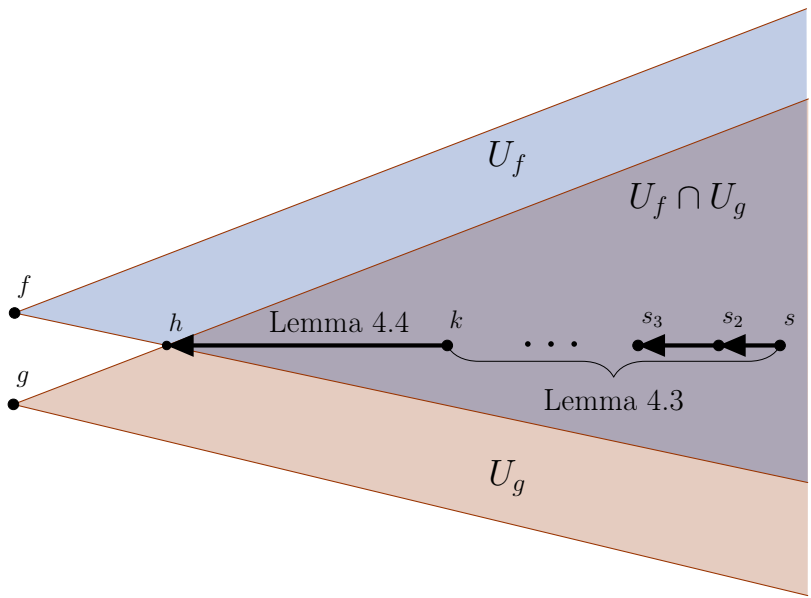
\begin{figure}
 \definecolor{uuuuuu}{rgb}{0,0,0}
\definecolor{zzttqq}{rgb}{0.6,0.2,0}
\definecolor{ttzzqq}{rgb}{0,0.2,0.6}
\definecolor{ududff}{rgb}{0,0,0}
\begin{tikzpicture}[line cap=round,line join=round,>=triangle 45,x=1cm,y=1cm,scale=0.8]
\begin{axis}[
x=1cm,y=1cm,
axis lines=none,
ymajorgrids=false,
xmajorgrids=false,
xmin=-9.1,
xmax=5,
ymin=-8,
ymax=3.5,
]
\clip(-12.25,-11.7) rectangle (12.25,7.14);
\fill[line width=2pt,color=zzttqq,fill=zzttqq,fill opacity=0.25] (5,1) -- (-8.21,-4.08) -- (4.99,-7.26) -- cycle;
\fill[line width=2pt,color=zzttqq,fill=ttzzqq,fill opacity=0.25] (4.97,2.5) -- (-8.19,-2.56) -- (4.99,-5.38) -- cycle;
\draw [line width=0.5pt,color=zzttqq] (5,1)-- (-8.21,-4.08);
\draw [line width=0.5pt,color=zzttqq] (-8.21,-4.08)-- (4.99,-7.26);
\draw [line width=0.5pt,color=zzttqq] (4.97,2.5)-- (-8.19,-2.56);
\draw [line width=0.5pt,color=zzttqq] (-8.19,-2.56)-- (4.99,-5.38);
\draw [->,line width=2pt] (4.53,-3.1) -- (3.51,-3.1);
\draw [->,line width=2pt] (3.51,-3.1) -- (2.17,-3.1);
\draw [->,line width=2pt] (-1,-3.1) -- (-5.663242884807682,-3.1);
\begin{scriptsize}
\draw [fill=ududff] (-8.21,-4.08) circle (2.5pt);
\draw[color=ududff] (-8.05,-3.65) node {$g$};
\draw [fill=ududff] (-8.19,-2.56) circle (2.5pt);
\draw[color=ududff] (-8.03,-2.13) node {$f$};
\draw [fill=uuuuuu] (-5.663242884807682,-3.1006263326891) circle (2pt);
\draw[color=uuuuuu] (-5.51,-2.71) node {$h$};
\draw [fill=ududff] (4.53,-3.1) circle (2.5pt);
\draw[color=ududff] (4.69,-2.67) node {$s$};
\draw [fill=ududff] (3.51,-3.1) circle (2.5pt);
\draw[color=ududff] (3.75,-2.67) node {$s_2$};
\draw [fill=ududff] (2.17,-3.1) circle (02.5pt);
\draw[color=ududff] (2.41,-2.67) node {$s_3$};
\draw [fill=ududff] (-1,-3.1) circle (2.5pt);
\draw[color=ududff] (-0.85,-2.67) node {$k$};
\draw[color=black] (0.8,-3.1) node {\huge $\cdots$};
\draw[color=black] (1.765,-4) node {\large{Lemma~\ref{lem:seq}}};
\draw[color=black] (-2.8,-2.7) node {\large{Lemma~\ref{lem:C2C1}}};
\draw[color=ududff] (0,0) node {\Large{$U_f$}};
\draw[color=ududff] (0,-5.3) node {\Large{$U_g$}};
\draw[color=ududff] (3,-0.7) node {\Large{$U_f \cap U_g$}};
\draw [decorate,decoration={brace,amplitude=15pt},xshift=0pt,yshift=0pt]
(4.53,-3.1) -- (-1,-3.1) node [black,midway,yshift=0.6cm] 
{};
\end{scriptsize}
\end{axis}
\end{tikzpicture}
\caption{\label{fig:1}Proof of Lemma~\ref{lem:main} and auxiliary results. Elements are ordered from left to right by $\prec$.}
\end{figure}

\subsection{Lower bilateral derivative and its basic properties}
In order to deal with a functions which are not necessarily differentiable we need to introduce certain generalization of derivative. For a function $f \colon I \to \R$ let us introduce the \emph{lower (bilateral) derivative}  \cite[p. 52]{Bru78}; \cite[Appendix I]{HabLalRou77} as 
\Eq{*}{
\LDer f(x_0):= \liminf_{x \to x_0}\frac{f(x)-f(x_0)}{x-x_0} \qquad (x_0 \in I).
}

Notice that every function $f$ has a (possibly infinite) lower derivative. 
Observe that for a lower derivative some 
some version of quotient rule remains is valid (see \eq{E:Diniquo} below); proof of this equality is similar to the standard one and therefore omitted.
Analogously chain rule holds whenever at least one (out of two) function is differentiable.
Finally, the following lemma holds.

\begin{lem}[\cite{HabLalRou77}, pp. 345--347] \label{lem:BilMon} Let $f \colon I \to \R$ be a continuous function. Then $f$ is nondecreasing if and only if $\LDer f (x)\ge 0$ for all $x \in I$.
\end{lem}

In the same way one can define the notation of upper derivative, which is however not used in this paper.

\subsection{Auxiliary results} 
In this section we will prove two technical results concerning comparability of means. Both of them are conncected with smoothness properties of quasi-arithmetic means. 

Let us emphasize that Lemma~\ref{lem:seq} is simply implied by Lemma~\ref{lem:main}, however it will be used in its proof. In particular our Lemma~\ref{lem:main} can be seamed as its strenghtening.
Contrary to this Lemma~\ref{lem:convex} and Lemma~\ref{lem:C2C1} are obviously a separated statements, which could be appliciable in a number of different settings.

\begin{lem}
 \label{lem:convex}
 Let $f,\,g \colon I \to \R$ be two continuous, strictly monotone functions, $a,b,c \in I$ with $a<b<c$. If $f|_{(a,b)} \prec g|_{(a,b)}$, $f|_{(b,c)} \prec g|_{(b,c)}$
 and both $f'(b)$ and $g'(b)$ exists and are nonzero then $f|_{(a,c)} \prec g|_{(a,c)}$.
\end{lem}
\begin{proof}
 Assume that both $f$ and $g$ are increasing. Applying \ref{CC:gf^-1} we obtain $g \circ f^{-1}$ is convex on $(f(a),f(b))$ and $(f(b),f(c))$. As $g \circ f^{-1}$ is differentiable at $f(b)$, it can be easily proved that $g \circ f^{-1}$ is convex at $(f(a),f(c))$ (as its one-sided derivatives are monotone). As the property \ref{CC:gf^-1} can be inversed, it implies $f|_{(a,c)} \prec g|_{(a,c)}$.
\end{proof}

\begin{lem}\label{lem:seq}
 Let $f,\,g \in \CM$ be two differentiable functions such that $f' \cdot g' \ne 0$, and $s \in U_f \cap U_g$. Then there exists a differentiable function $k \in U_f \cap U_g$ such that $k'$ is nonvanishing and $k \prec s$.
 
\end{lem}
Idea of the prove is to correct a function $s$ iteratively, removing the nondiferentiability point one-by-one, and finally pass to the limit.
\begin{proof} 
Take $s \in U_f \cap U_g$ arbitrarily. Assume without loss of generality that $f,\,g,\,s$ are all increasing.

By Lemma~\ref{lem:hnonvanish} we know that $s$ is one-sided differentiable at every point and, moreover, its one-sided derivatives are nowhere vanishing. In particular, as $s$ is convex with respect to $f$, we get that there exists a countable set $Z=(z_n)_{n \in \N}$, such that $s$ has a derivative at every point of $I \setminus Z$. Relative convexity of $s$ also implies that 
\Eq{*}{
s'_-(z_n) \le s'_+(z_n) \quad \text{ for all } z_n \in Z.
}
Set $s_1:=s$ and define, for all $n \ge 1$,
\Eq{*}{
s_{n+1}(x):=\begin{cases}
            \dfrac{(s_n)'_+(z_n)}{(s_n)'_-(z_n)}(s_{n}(x)-s_{n}(z_n))+s_{n}(z_n) & \text{ for }x <z_n,\\[4mm]
            s_n(x) & \text{ for }x \ge z_n.
           \end{cases}
}
Then the mapping $n \mapsto s_n(x)$ is pointwise decreasing for all $x \in I$ and $s_{n_0}$ is differentiable at $I \setminus (z_n)_{n \ge n_0}$, for all $n_0 \in \N$.

Suppose that $s_n \succ f$. By \ref{CC:gf^-1}, as $s_n$ is increasing, $s_{n} \circ f^{-1}$ is convex on $f(I)$. Thus $s_{n+1} \circ f^{-1}$ is convex on $f\big(I\cap (-\infty,x_{n}) \big)$ and (separately) on 
$f\big(I\cap (x_{n},\infty) \big)$. Furthermore
\Eq{*}{
(s_{n+1} \circ f^{-1})'_+(f(z_n))&=(s_n)'_+(z_n)\cdot (f^{-1})'(z_n),\\
(s_{n+1} \circ f^{-1})'_-(f(z_n))&=\frac{(s_n)'_+(z_n)}{(s_n)'_-(z_n)} \cdot (s_n)'_-(z_n)(f^{-1})'(z_n)\\
&=(s_n)'_+(z_n)(f^{-1})'(z_n).
}
Thus $s_{n+1} \circ f^{-1}$ is differentiable at $f(z_n)$ and, in view of Lemma~\ref{lem:convex}, convex. It implies $f \prec s_{n+1}$, similarly $g \prec s_{n+1}$. Thus $s_{n+1} \in U_f \cap U_g$. 

Additionally we have $s_{n+1} \le s_{n}$ for all $n \in \N$.
Let $k \colon I \to \R$ be a pointwise limit of the sequence $(s_n)_{n \in \N}$, i.e. $k:= \lim_{n \to \infty} s_n$. 

As $s_n \circ f^{-1}$ is convex for every $n \in \N$, its pointwise limit $k \circ f^{-1}$ is convex too. In particular $k \circ f^{-1}$ is continuous function defined of $f(I)$. Therefore $k$ being a composition of two continuous functions is continuous too. 

Similarly $f \circ s_n^{-1}$ is concave for every $n \in \N$ and pointwise monotone (as a function of $n$). In particular it has a continuous limit $u$.
Then $k^{-1}=f^{-1} \circ u$, consequently $k$ is invertible. Therefore, as it $k$ also continuous, we obtain that it is strictly monotone. 

It implies that $k$ generates a quasi-arithmetic mean. Moreover, in view of Lemma~\ref{lem:Pales}, we have $\lim_{n \to \infty} \QA{s_n}=\QA{k}$.

Furthermore, by Lemma~\ref{lem:convex}, for all $n \in \N$, the function 
\Eq{*}{
s_{n+1} \circ s_n^{-1}(y)=
\begin{cases}
            y & \text{ for }y \ge s_n(x_n),\\
            \dfrac{(s_n)'_+(x_n)}{(s_n)'_-(x_n)}(y-s_{n}(x_n))+s_{n}(x_n) & \text{ for }y <s_n(x_n)\\
           \end{cases}
}
is concave. In view of \ref{CC:fg^-1}, it imples  $s_{n+1} \prec s_n$ for all $n \in \N$.
 In particular 
 \Eq{*}{
 k=\lim_{n \to \infty} s_n \prec s_{n_0} \text{ for all }n_0 \in \N.
 }
 Moreover as $s_n \in U_f \cap U_g$ for all $n \in \N$, then $k$ being a limit of $s_n$ is an element of $U_f \cap U_g$ too. Furthermore $k \prec s_1=s$. 
 
Notice that in view of Lemma~\ref{lem:diffnonz} for all $n_0 \in \N$, inequality $f \prec k \prec s_{n_0}$ implies that $k$ is differentiable at $I \setminus \{z_n\}_{n \ge n_0}$. As $n_0$ was taken arbitrarily we get that  $k$ is differentiable. Moreover, by Corollary~\ref{cor:hnonvanish}, as $f'$ is nowhere vanishing, so is $k'$.
\end{proof}

Next lemma provides a generalization of Mikusi\'nski comparability condition \ref{CC:Mikusinski} under a mixed smoothness assumption (one function is $\mathcal{C}^1$, second one is $\mathcal{C}^2$). Notice that, as assumptions of generating function are different, one can expect some dual results using the symmetry of quasi-arithmetic means described for example in \cite{PalPas18b}. Indeed, this is a case in this lemma. However, just to keep compactness of its wording it is omitted.

\begin{lem}
\label{lem:C2C1}
 Assume that $f,g \in \CM$. If $f$ is a $\mathcal{C}^2$ function and $k$ is an increasing $\mathcal{C}^1$ function, moreover $f'\cdot k' \ne 0$. Then $\QA{f} \le \QA{k}$ if and only if
\Eq{*}{
\dfrac{f''(x)}{f'(x)} \le \dfrac{\LDer(k')(x)}{k'(x)}\qquad \text{ for all }x \in I.
}\end{lem}

\begin{proof}
 Assume without loss of generality that $f$ is increasing.

Therefore, as $\QA{f} \le \QA{k}$ we get, by \ref{CC:g'/f'}, $k'/f'$ is nondecreasing. If we repeat the prove of quotient rule to lower derivative (recall that $f'$ is differentiable by the assumption) and apply Lemma~\ref{lem:BilMon} we obtain
\Eq{E:Diniquo}{
0 \le \LDer \Big(\frac{k'}{f'}\Big) (x)=\frac{\LDer(k')(x)f'(x)-f''(x)k'(x)}{f'(x)^2},
}
which simplifies to $
\LDer(k')(x)/k'(x) \ge {f''(x)}/{f'(x)}$.

As all implications were in fact equivalence we obtain that it is both necessary and sufficient condition.
\end{proof}

\subsection{\label{ssec:thm:main} Proof of Lemma~\ref{lem:main}}
First observe that there exists a function $h \in \mathcal{C}^2(I)$ sattisfying \eq{def:h}, as the differential equation $h''/h'=v$ has a solution for every continuous function $v$. Furthermore $h$ is defined up to an affine transformation (cf. \cite{Pas13} for details), in particular all solutions of \eq{def:h} generate the same quasi-arithmetic mean.

It is easy to verify that $h \in U_f \cap U_g$. We will show that in fact $h$ is the smallest element in $U_f \cap U_g$.

Take $s \in U_f \cap U_g$ arbitrarily.
Applying Lemma~\ref{lem:seq}, there exists $k \in U_f \cap U_g$ such that $k \prec s$, $k$ is differentiable, and $k' \ne 0$. Assume that $f$, $g$, and $k$ are increasing.

If we apply Lemma~\ref{lem:C2C1} for the pairs $(f,\,k)$ and $(g,\,k)$ and use the  max function we obtain 
\Eq{*}{
\frac{\LDer(k')(x)}{k'(x)} &\ge \max \left( \frac{f''(x)}{f'(x)},\frac{g''(x)}{g'(x)}\right)=\frac{h''(x)}{h'(x)}\qquad (x \in I)\:.
}
Therefore, applying the oposite implication in Lemma~\ref{lem:C2C1}, we get $h \prec k$. As we also know that $k \prec s$, we get
\Eq{*}{
h \prec s \quad \text{ for all }\quad s \in U_f \cap U_g,
}
which implies that $h$ is the smallest element of $U_f \cap U_g$ (more precisely $[h]_\sim$ is the smallest element in the quotient set $(U_f \cap U_g) /_\sim$ embed with the induced partial ordering). This property is trivially equivalent to the statement of this lemma.

\subsection*{Acknowledgement} I am grateful to professor Roman Badora who suggest me to investigate lattice properties for means.

\def\cprime{$'$} \def\R{\mathbb R} \def\Z{\mathbb Z} \def\Q{\mathbb Q}
  \def\C{\mathbb C}

\end{document}